\documentclass[oneside,english]{amsart}

\usepackage[T1]{fontenc}
\usepackage[latin9]{inputenc}
\usepackage{textcomp}
\usepackage{amsthm}
\usepackage{amstext}
\usepackage{amssymb}
\usepackage[all]{xy}

\makeatletter

\newcommand{\lyxmathsym}[1]{\ifmmode\begingroup\def\b@ld{bold}
  \text{\ifx\math@version\b@ld\bfseries\fi#1}\endgroup\else#1\fi}

\numberwithin{equation}{section}
\numberwithin{figure}{section}
\theoremstyle{plain}
\newtheorem{thm}{\protect\theoremname}
  \theoremstyle{plain}
  \newtheorem{lem}[thm]{\protect\lemmaname}
  \theoremstyle{plain}
  \newtheorem{cor}[thm]{\protect\corollaryname}
  \theoremstyle{plain}
  \newtheorem{prop}[thm]{\protect\propositionname}

\makeatother

\usepackage{babel}
  \providecommand{\corollaryname}{Corollary}
  \providecommand{\lemmaname}{Lemma}
  \providecommand{\propositionname}{Proposition}
\providecommand{\theoremname}{Theorem}

\title[On some subvarieties]{On some subvarieties of the cartesian powers of a semisimple algebra related to a parabolic subalgebra}

\author{Mouchira Zaiter}

\address{Universit\'e Paris 7 - CNRS \\
Institut de Math\'ematiques de Jussieu \\
Th\'eorie des groupes \\
Case 7012 \\ B\^atiment Chevaleret \\
75205 Paris Cedex 13, France}

\email{zaiter@math.jussieu.fr}
 
\begin{document}

\date\today
\maketitle

\begin{abstract}
In this note, we describe some desingularizations of some subvarieties of the
cartesian powers of a semisimple Lie algebra of finite dimension.
\end{abstract}
\maketitle
\setcounter{tocdepth}{1}
\tableofcontents

\section{Introduction}

The basic field $\Bbbk$ is an algebraic closed field of characteristic
zero. Let $\mathfrak{g}$ be a semisimple Lie algebra of finite dimension
and let $G$ be its adjoint group. Let $\mathfrak{b}$ be a Borel
subalgebra of $\mathfrak{g}$, $\mathfrak{h}$ a Cartan subalgebra
of $\mathfrak{g}$ contained in $\mathfrak{b}$ and let $N_{G}\left(\mathfrak{h}\right)$
be the normalizer of $\mathfrak{h}$ in $G$. Let denote by $\mathfrak{p}$
a parabolic subalgebra of $\mathfrak{g}$ containing $\mathfrak{b}$,
$\mathfrak{l}$ its reductive factor containing $\mathfrak{h}$, $\mathfrak{p}_{\mathrm{u}}$
its nilpotent radical and $\mathbf{P}$ its normalizer in $G$.
Let set:
\[
\begin{array}{c}
\mathcal{P}^{\left(k\right)}:=\left\{ \left(x_{1},\ldots,x_{k}\right)\in\mathfrak{g}^{k}\mid\exists g\in G\textrm{ such that }\left(g\left(x_{1}\right),\ldots,g\left(x_{k}\right)\right)\in\mathfrak{p}^{k}\right\} ,\\
\mathcal{P}_{\mathfrak{\mathrm{u}}}^{\left(k\right)}:=\left\{ \left(x_{1},\ldots,x_{k}\right)\in\mathfrak{g}^{k}\mid\exists g\in G\textrm{ such that }\left(g\left(x_{1}\right),\ldots,g\left(x_{k}\right)\right)\in\mathfrak{p}_{\mathrm{u}}^{k}\right\} .
\end{array}
\]
For $X$ a $G$-variety, we denote by $G\times_{\mathbf{P}}X$ the
quotient of $G\times X$ under the right action of $\mathbf{P}$ given
by $\left(g,x\right).p:=\left(gp,p^{-1}.x\right)$. The main result of this note is the following theorem:
\begin{thm}
\label{thm1}(i) For $k\geqslant2$, $G\times_{\mathbf{P}}\mathfrak{p}^{k}$
is a desingularization of $\mathcal{P}^{\left(k\right)}$. Moreover,
$\mathcal{P}^{\left(k\right)}$ is not normal.

(ii) For  $k\geqslant2$, the canonical morphism $G\times_{\mathbf{P}}\mathfrak{p}_{\mathrm{u}}^{k}\rightarrow\mathcal{P}_{\mathrm{u}}^{\left(k\right)}$ is projective and factorizes through a desingularization of an affine variety $X^{\left(k\right)}$ and the morphism $X^{\left(k\right)}\rightarrow\mathcal{P}_{\mathrm{u}}^{\left(k\right)}$ is finite and pure in the sense of \cite{key-11}.
\end{thm}
For $\mathfrak{p=\mathfrak{b}},$ this theorem is given in \cite{key-1},
a joint work with J-Y. Charbonnel. Moreover, in this case, $\mathcal{P}_{\mathrm{u}}^{\left(k\right)}$
is normal and has a rational singularities. In fact, this note is the first step to a generalization of some results of \cite{key-1}.

Let $S\left(\mathfrak{g}\right)$ and $S\left(\mathfrak{h}\right)$
be the symmetric algebras of $\mathfrak{g}$ and $\mathfrak{h}$ respectively
and let $S\left(\mathfrak{h}\right)^{W}$ be the subalgebra of $W$-invariant
elements of $S\left(\mathfrak{h}\right)$, where $W$ is the Weyl
group of $\mathfrak{g}$ with respect to $\mathfrak{h}$. In section
3, we use a variety $\chi_{0}$ introduced in \cite{key-1}, which is the closed subvariety of $\mathfrak{g}\times\mathfrak{h}$ such
that its algebra of regular functions $\Bbbk\left[\chi_{0}\right]$
equals $S\left(\mathfrak{g}\right)\otimes_{S\left(\mathfrak{h}\right)^{W}}S\left(\mathfrak{h}\right)$,
and we introduce the varieties $\chi:=\chi_{0}//W_{\mathfrak{l}}$,
where $W_{\mathfrak{l}}$ is the Weyl group of the Levi factor $\mathfrak{l}$
of $\mathfrak{p}$ with respect to $\mathfrak{h}$. We prove that
$G\times_{\mathbf{P}}\mathfrak{p}^{k}$ is a desingularization of a subvariety
$\chi^{\left(k\right)}$ of $\chi^k$.

In section 4, we introduce the variety $X:=Spec\mathcal{A}$ and $X^{\left(k\right)}$
a subvariety of $X^{k}$, where $\mathcal{A}$ is the integral closure
of the algebra of regular functions $\Bbbk\left[G\left(\mathfrak{p}_{\mathrm{u}}\right)\right]$
of $G\left(\mathfrak{p}_{\mathrm{u}}\right)$ in the field of rational
functions $\Bbbk\left(G\times_{\mathbf{P}}\mathfrak{p}_{\mathrm{u}}\right)$
of $G\times_{\mathbf{P}}\mathfrak{p}_{\mathrm{u}}$ and we prove that
$G\times_{\mathbf{P}}\mathfrak{p}_{\mathrm{u}}^{k}$ is a desingularization
of $X^{\left(k\right)}$.

\section{Notations}

We consider the diagonal action of $G$ on $\mathfrak{g}^{k}$.  Let $\mathcal{R}$
be the root system of $\mathfrak{h}$ in $\mathfrak{g}$ and let $\mathcal{R}_{+}$
be the positive root system of $\mathcal{R}$ defi{}ned by $\mathfrak{b}$.
The Weyl group of $\mathcal{R}$ is denoted by $W$ and the basis
of $\mathcal{R}_{+}$ is denoted by $\Pi$. Let $\Omega$ be the Richardson orbit of $\mathfrak{p}$ and let $\mathfrak{p}'_{\mathrm{u}}:=\Omega\cap\mathfrak{p}_{\mathrm{u}}$. Set: 
\[
\begin{array}{c}
\mathcal{R}_{\mathfrak{l}}:=\left\{ \alpha\in\mathcal{R}\mid\mathfrak{g}^{\alpha}\subset\mathfrak{l}\right\} \\
\mathfrak{p}_{-}:=\mathfrak{l}\oplus\bigoplus_{\alpha\in\mathcal{R}_{+}\setminus\mathcal{R}_{\mathfrak{l}}}\mathfrak{g}^{-\alpha}.
\end{array}
\]
 Let $W_{\mathfrak{l}}$ be the Weyl group of $\mathcal{R}_{\mathfrak{l}}$ and let  $\mathbf{P}_-$ be the normalizer of $\mathfrak{p}_-$ in $G$. Denote
by $\mathbf{L}$ the identity component of the normalizer of $\mathfrak{l}$ in $G$ and denote by $\mathbf{P}_{\mathrm{u}}$ and $\mathbf{P}_{-,\mathrm{u}}$ the unipotent radicals of $\mathbf{P}$ and $\mathbf{P}_-$ respectively.
 We use the following notations: 
\begin{itemize}
\item All topological terms refer to the Zariski topology. For $Y$ an open
subset of the algebraic variety $X$, $Y$ is called a big open subset
if the codimension of $X\setminus Y$ in $X$ is at least $2$. The algebra of regular functions on $X$ is denoted by $\Bbbk[X]$. When $X$ is irreducible the field of rational functions on $X$ is denoted by $\Bbbk\left(X\right)$.
\item if $G\times A\longrightarrow A$ is an action of $G$ on the algebra
$A$, denote by $A^{G}$ the subalgebra of the $G$-invariant elements
of $A,$ 
\item for $i=1,\ldots,\mathrm{rk}\mathfrak{g}$, the 2-order polarizations
of $p_{i}$ of bidegree $(d_{i}-n,n)$, denoted by $p_{i}^{(n)}$, are
the unique elements in $(S(\mathfrak{g})\otimes_{\mathbb{C}}S(\mathfrak{g}))^{G}$
satisfying the following relation 
\[
p_{i}(ax+by)=\sum\limits _{n=0}^{d_{i}}a^{d_{i}-n}b^{n}p_{i}^{(n)}(x,y),
\]
 for all $a,b\in\mathbb{C}$ and for all $(x,y)\in\mathfrak{g}\times\mathfrak{g}$, 
\item $<.,.>$ is the Killing form of $\mathfrak{g}$, 
\item for $i\in\left\{ 1,\ldots,\mathrm{rk}\mathfrak{g}\right\} $, $\varepsilon_{i}$
is the element of $S(\mathfrak{g})\otimes_{\mathbb{C}}\mathfrak{g}$
defined by 
\[
<\varepsilon_{i}(x),v>=p_{i}'(x)(v),\mbox{ }\forall x,v\in\mathfrak{g},
\]
 where $p_{i}'(x)$ is the differential of $p_{i}$ at $x$ for all
$i\in\left\{ 1,\ldots\mathrm{rk}\mathfrak{g}\right\} $, 
\item for $i\in\left\{ 1,\ldots,\mathrm{rk}\mathfrak{g}\right\} $ and for
$m\in\left\{ 0,\ldots,d_{i}-1\right\} $, the $2$-polarizations of
$\varepsilon_{i}$ of bidegree $(d_{i}-m-1,m)$ denoted by $\varepsilon_{i}^{(m)}$
are the unique elements in\linebreak $S(\mathfrak{g})\otimes_{\mathbb{C}}S(\mathfrak{g})\otimes_{\mathbb{C}}\mathfrak{g}$
satisfying the following relation 
\[
\varepsilon_{i}(ax+by)=\sum\limits _{n=0}^{d_{i}-1}a^{d_{i}-n-1}b^{n}\varepsilon_{i}^{(n)}(x,y),
\]
 for all $a,b\in\mathbb{C}$ and $(x,y)\in\mathfrak{g}\times\mathfrak{g}$,
\item for $(x,y)\in\mathfrak{g}\times\mathfrak{g}$, $V_{x,y}$ is the space
generated by the set 
\[
\left\{ \varepsilon_{i}^{(m)}\left(x,y\right),i\in\left\{ 1,\ldots\mathrm{rk}\mathfrak{g}\right\} ,m\in\left\{ 0,\ldots,d_{i}-1\right\} \right\} ,
\]

\item $\mathfrak{g}_{\mathrm{reg}}$ is the set of regular elements of $\mathfrak{g}$,
\item for $\mathfrak{a}$ a subalgebra of $\mathfrak{g}$, $\mathfrak{a}_{\mathrm{reg}}:=\mathfrak{a}\cap\mathfrak{g}_{\mathrm{reg}}$, 
\item for $x$ in $\mathfrak{g}$, $G^{x}$ is the centralizer of $x$ in
$G$,
\item for $\mathfrak{a}$ a subalgebra of $\mathfrak{g}$ and for $x$ in
$\mathfrak{a}$, $\mathfrak{a}^{x}$ is the centralizer of $x$ in
$\mathfrak{a}$,
\item for $(x_1,\ldots,x_k)\in\mathfrak{g}^k$, $P_{(x_1,\ldots,x_k)}$ is the space
of $\mathfrak{g}$ generated by $x_1,\ldots,x_k$,
\item $\varpi$ is the canonical morphism from $\mathfrak{p}$ to $\mathfrak{l}$,
\item for $x$ in $\mathfrak{p}$, $\tilde{x}$ is the image of $x$ by
$\varpi$,
\item $\Omega_{\mathfrak{g}}:=\left\{ \left(x,y\right)\in\mathfrak{g\times\mathfrak{g}}\mid P_{x,y}\backslash\left\{ 0\right\} \subset\mathfrak{g}_{\mathrm{reg}}\textrm{ and }\mathrm{dim}\, P_{x,y}=2\right\} $,
\item $\left(e,h,f\right)$ is a principal $\mathfrak{sl}_{2}$-triplet
\item for $X$ an algebraic variety, $\mathcal{O}_{X}$ is its structural
sheaf, $\Bbbk[X]$ is the algebra of regular functions on $X$ and
$\Bbbk(X)$ is the fi{}eld of rational functions on $X$ when $X$
is irreducible.
\end{itemize}
The following Lemmas are well known:
\begin{lem}\label{Lem1.4}
Let $P$ and $Q$ be parabolic subgroups of $G$ such that $P$ is
contained in $Q$. Let $X$ be a $Q$-variety and let $Y$ be a closed
subset of $X$, invariant under $P$. Then $Q.Y$ is a closed subset
of $X$. Moreover, the canonical map from $Q\times_PY$ to $Q.Y$
is a projective morphism. \end{lem}
\begin{proof}
Since $P$ and $Q$ are parabolic subgroups of $G$ and since $P$
is contained in $Q$, $Q/P$ is a projective variety. Let denote by
$Q\times_PX$ and $Q\times_PY$ the quotients of $Q\times X$
and $Q\times Y$ under the right action of $P$ given by $(g,x).p:=(gp,p^{-1}.x)$.
Let $g\mapsto\bar{g}$ be the quotient map from $Q$ to $Q/P$. Since
$X$ is a $Q$-variety, the map $$Q\times X\rightarrow Q/P\times X\atop(g,x)\mapsto(\bar{g},g.x)$$
defines through the quotient an isomorphism
from $Q\times_{P}X$ to $Q/P\times X$. Since $Y$ is a $P$-invariant
closed subset of $X$, $Q\times_{P}Y$ is a closed subset of $Q\times_{P}X$
and its image by the above isomorphism equals $Q/P\times Q.Y$. Hence
$Q.Y$ is a closed subset of $X$ since $Q/P$ is a projective variety.
From the commutative diagram 
\[
\xymatrix{Q\times_{P}Y\ar[r]\ar[rd] & Q/P\times Q.Y \ar[d]\\
 & Q.Y
}
\]
one deduces that the map $Q\times_{P}Y\rightarrow Q.Y$
is a projective morphism.\end{proof}

{\bf Acknowledgements}. I would like to thank Jean-Yves Charbonnel for his advice and help and for his rigorous attention to this work.

\section{On the variety $\mathcal{P}^{\left(k\right)}$}

\subsection{On parabolic subalgebras}
\begin{lem}
\label{lem:Vxy in a}Let $\mathfrak{a}$ be an algebraic subalgebra
of $\mathfrak{g}$. 

(i) Let suppose that $\mathfrak{a}$ contains $\mathfrak{g}^{x}$
for all $x$ in a nonempty open subset of $\mathfrak{a}$ and let
suppose that $\mathfrak{\mathfrak{a}}_{\mathrm{reg}}$ is not empty.
Then $V_{x,y}$ is contained in $\mathfrak{a}$ for all $\left(x,y\right)$
in $\mathfrak{a}\times\mathfrak{a}$.

(ii) Let suppose that $\mathfrak{a}$ contains a Cartan subalgebra
of $\mathfrak{g}$. Then $V_{x,y}$ is contained in $\mathfrak{a}$
for all $\left(x,y\right)$ in $\mathfrak{a}\times\mathfrak{a}$.\end{lem}
\begin{proof}
(i) By hypothesis, for all $x$ in a nonempty open subset of $\mathfrak{a}$,
$x$ is a regular element and $\mathfrak{g}^{x}$ is contained in
$\mathfrak{a}$. So by \cite{key-4} Theorem 9, $\varepsilon_{1}\left(x\right),\ldots,\varepsilon_{l}\left(x\right)$
belong to $\mathfrak{a}$ for all $x$ in a nonempty open subset of
$\mathfrak{a}$ and hence for all $x$ in $\mathfrak{a}$ by continuity.
As a result, for all $\left(x,y\right)$ in $\mathfrak{a}\times\mathfrak{a}$,
$\left\{ \varepsilon_{i}^{\left(m\right)}\left(x,y\right),i\in\left\{ 1,\ldots,l\right\} ,m\in\left\{ 0,\ldots,d_{i}-1\right\} \right\} $
is contained in $\mathfrak{a}$, whence the assertion.

(ii) Let $\mathfrak{c}$ be a Cartan subalgebra of $\mathfrak{g}$
contained in $\mathfrak{a}$. Since $\mathfrak{a}$ is an algebraic
subalgebra of $\mathfrak{g}$, all semisimple element of $\mathfrak{a}$
is conjugate under the adjoint group of $\mathfrak{a}$ to an element
of $\mathfrak{c}$. Hence for all regular semisimple element of $\mathfrak{g}$,
belonging to $\mathfrak{a}$, $\mathfrak{g}^{x}$ is contained in
$\mathfrak{a}$. As a result, the assertion is a consequence of (i)
since the subset of regular semisimple elements of $\mathfrak{g}$,
belonging to $\mathfrak{a}$, is a nonempty open subset of $\mathfrak{a}$.\end{proof}
\begin{cor}
\label{cor:Vxy=00003Db}For all $\left(x,y\right)$ in $\mathfrak{p\times p}$,
$V_{x,y}$ is contained in $\mathfrak{p}$. In particular, for all
$\left(x,y\right)$ in a nonempty open subset of $\mathfrak{b\times b}$,
$V_{x,y}=\mathfrak{b}$.\end{cor}
\begin{proof}
Since $\mathfrak{h}$ is contained in $\mathfrak{p}$, for all $\left(x,y\right)$
in $\mathfrak{p}\times\mathfrak{p}$, $V_{x,y}$ is contained in $\mathfrak{p}$
by Lemma \ref{lem:Vxy in a}, (ii). Since $\left(h,e\right)$ belongs
to $\Omega_{\mathfrak{g}}$, $\Omega_{\mathfrak{g}}\cap\mathfrak{b}\times\mathfrak{b}$
is a nonempty open subset, whence the corollary.\end{proof}
\begin{lem}
\label{Lm:Vxyc}Set:
\[
\begin{array}{ccc}
R_{\mathfrak{p}} & := & \left\{ x\in\mathfrak{g}_{\mathrm{reg}}\cap\mathfrak{p}|\varpi\left(x\right)\in\mathfrak{l}_{\mathrm{reg}}\right\} \\
R'_{\mathfrak{p}} & := & \left\{ x\in R_{\mathfrak{p}}\mid\mathfrak{g}^{x}\cap\mathfrak{p}_{\mathrm{u}}=\left\{ 0\right\} \right\} .
\end{array}
\]

(i) For all $x$ in $R_{\mathfrak{p}}$, $\varpi\left(\mathfrak{g}^{x}\right)=\mathfrak{l}^{\varpi\left(x\right)}$
if and only if $\mathfrak{g}^{x}\cap\mathfrak{p}_{\mathrm{u}}=\left\{ 0\right\} $.

(ii) The subset $R'_{\mathfrak{p}}$ is open in $\mathfrak{p}$.

(iii) For all $\left(x,y\right)$ in $\mathfrak{p}\times\mathfrak{p}$,
$V_{x,y}$ is contained in $V_{\varpi\left(x\right),\varpi\left(y\right)}^{\mathfrak{l}}+\mathfrak{p}_{\mathrm{u}}$.

(iv) For all $\left(x,y\right)$ in $R'_{\mathfrak{p}}\times\mathfrak{p}$,
$\varpi\left(V_{x,y}\right)=V_{\varpi\left(x\right),\varpi\left(y\right)}^{\mathfrak{l}}$.\end{lem}
\begin{proof}
(i) Let $x$ be in $R_{\mathfrak{p}}$. By Lemma \ref{lem:Vxy in a}
(ii), $\mathfrak{g}^{x}$ is a Cartan subalgebra contained in $\mathfrak{p}$.
Since $\varpi$ is a morphism of Lie algebra, $\varpi\left(\mathfrak{g}^{x}\right)$
is contained in $\mathfrak{l}^{\varpi\left(x\right)}$. Furthermore,
$\dim\varpi\left(\mathfrak{g}^{x}\right)=\mathrm{rk}\mathfrak{g}$
if and only if $\mathfrak{g}^{x}\cap\mathfrak{p}_{\mathrm{u}}=\left\{ 0\right\} $,
whence the assertion since $\mathfrak{l}$ has rank $\mathrm{rk}\mathfrak{g}$.

(ii) For $x$ regular semisimple in $\mathfrak{p}$, $\mathfrak{g}^{x}$
is a Cartan subalgebra contained in $\mathfrak{p}$. As a result,
$\mathfrak{g}^{x}\cap\mathfrak{p}_{\mathrm{u}}=\left\{ 0\right\} $.
So by (i), $R_{\mathfrak{p}}$ and $R'_{\mathfrak{p}}$ are not empty
since for such a $x$, $\varpi\left(x\right)$ is a regular semisimple
element of $\mathfrak{l}$. The subset $R_{\mathfrak{p}}$ of $\mathfrak{p}$
is open since the subsets of regular elements of $\mathfrak{g}$ and
$\mathfrak{l}$ are open in $\mathfrak{g}$ and $\mathfrak{l}$ respectively.
The map $x\mapsto\mathfrak{g}^{x}$ from $R_{\mathfrak{p}}$ to the Grassmanian $\mathrm{Gr}_{\mathrm{rk}\mathfrak{g}}\left(\mathfrak{g}\right)$
is regular. So $R'_{\mathfrak{p}}$ is an open subset of $\mathfrak{p}$.

(iii) Let $L_{\mathfrak{l}}$ be the submodule of elements $\varphi$
of $\mathrm{S\left(\mathfrak{l}\right)\otimes_{\Bbbk}}\mathfrak{l}$
such that $\left[\varphi\left(x\right),x\right]=0$, for all $x$
in $\mathfrak{l}$. Then $L_{\mathfrak{l}}$ is a free module of rank
$\mathrm{rk}\mathfrak{g}$ according to \cite{key-3}. Denote by $\varphi_{1},\ldots,\varphi_{\mathit{l}}$
a basis of $L_{\mathfrak{l}}$ and denote by $R_{\mathfrak{p,l}}$
the subset of elements $x$ of $\mathfrak{p}$ such that $\varpi\left(x\right)$
is in $\mathfrak{l}_{\mathrm{reg}}$. According to \cite{key-13},
$\mathfrak{l}_{\mathrm{reg}}$ is a big open subset of $\mathfrak{l}$.
So $R_{\mathfrak{p,l}}$ is a big open subset of $\mathfrak{p}$.
For $x$ in $R_{\mathfrak{p,l}}$ and for $i=1,\ldots,\mathit{l}$,
$\varpi\circ\varepsilon_{i}\left(x\right)$ belongs to $\mathfrak{l}^{\varpi\left(x\right)}$.
So there exists a unique element $\left(a_{i,1}\left(x\right),\ldots,a_{i,\mathit{l}}\left(x\right)\right)$
of $\Bbbk^{\mathit{l}}$ such that 
\[
\varpi\circ\varepsilon_{i}\left(x\right)=a_{i,1}\left(x\right)\varphi_{1}\circ\varpi\left(x\right)+\cdots+a_{i,\mathit{l}}\left(x\right)\varphi_{\mathit{l}}\circ\varpi\left(x\right).
\]
 The functions $a_{i,1},\ldots,a_{i,\mathit{l}}$
so defined on $R_{\mathfrak{p,l}}$ have regular extensions to $\mathfrak{p}$
since $\mathfrak{p}$ is normal and since $R_{\mathfrak{p,l}}$ is
a big open subset of $\mathfrak{p}$. As a result, for all $\left(x,y\right)$
in $\mathfrak{p}\times\mathfrak{p}$ and for all $\left(a,b\right)$
in $\Bbbk^{2}$, $\varpi\circ\varepsilon_{i}\left(ax+by\right)$ is
a linear combination of the elements $\varphi_{1}\left(ax+by\right),\ldots,\varphi_{\mathit{l}}\left(ax+by\right)$.
Hence, by \cite{key-7}, $\varpi\left(V_{x,y}\right)$ is contained
in $V_{\varpi\left(x\right),\varpi\left(y\right)}^{\mathfrak{l}}$
for all $\left(x,y\right)$ in $\mathfrak{p}\times\mathfrak{p}$,
whence the assertion.

(iv) Let $\left(x,y\right)$ be in $R'_{\mathfrak{p}}\times\mathfrak{p}$.
For all $z$ in a nonempty open subset of $P_{x,y}$, $z$ belongs
to $R'_{\mathfrak{p}}$ since $x$ belongs to $R'_{\mathfrak{p}}$.
So by (i), $\mathfrak{l}^{\varpi\left(z\right)}$ is contained in
$\varpi\left(V_{x,y}\right)$ for all $z$ in a nonempty open subset
of $P_{x,y}$. As a result, according to \cite{key-7}, $V_{\varpi\left(x\right),\varpi\left(y\right)}^{\mathfrak{l}}$
is contained in $\varpi\left(V_{x,y}\right)$, whence the assertion
by (iii).\end{proof}
\begin{cor}
\label{cor:Vxy=00003DVl+pu}For all $\left(x,y\right)$ in $\Omega_{\mathfrak{g}}\cap\mathfrak{p}\times\mathfrak{p}$,
$V_{x,y}=V_{\varpi\left(x\right),\varpi\left(y\right)}^{\mathfrak{l}}+\mathfrak{p}_{\mathrm{u}}$.\end{cor}
\begin{proof}
Since $\left(h,e\right)$ belongs to $\mathfrak{p}\times\mathfrak{p}$,
$\Omega_{\mathfrak{g}}\cap\mathfrak{p}\times\mathfrak{p}$ is a nonempty
open subset of $\mathfrak{p}\times\mathfrak{p}$. Let $\left(x,y\right)$
be in $\Omega_{\mathfrak{g}}\cap R'_{\mathfrak{p}}\times\mathfrak{p}$.
By Lemma \ref{Lm:Vxyc} (iv), $\varpi\left(V_{x,y}\right)=V_{\varpi\left(x\right),\varpi\left(y\right)}^{\mathfrak{l}}$.
Furthermore, $\dim V_{x,y}=\mathrm{b}_{\mathfrak{g}}$ since $\left(x,y\right)$
belongs to $\Omega_{\mathfrak{g}}$. Hence $\mathfrak{p}_{\mathrm{u}}$
is contained in $V_{x,y}$ and $\dim V_{\varpi\left(x\right),\varpi\left(y\right)}^{\mathfrak{l}}=\mathrm{b}_{\mathfrak{l}}$
since $\mathrm{b}_{\mathfrak{g}}=\mathrm{b}_{\mathfrak{l}}+\dim\mathfrak{p}_{\mathrm{u}}$.
According to Lemma \ref{lem:Vxy in a} (ii), the map $\left(x,y\right)\mapsto V_{x,y}$
is a regular map from $\Omega_{\mathfrak{g}}\cap\mathfrak{p}\times\mathfrak{p}$
to $\mathrm{Gr}_{\mathrm{b}_{\mathfrak{g}}}\left(\mathfrak{p}\right)$.
So for all $\left(x,y\right)$ in $\Omega_{\mathfrak{g}}\cap\mathfrak{p}\times\mathfrak{p}$,
$\mathfrak{p}_{\mathrm{u}}$ is contained in $V_{x,y}$ and $\dim\varpi\left(V_{x,y}\right)=\mathrm{b}_{\mathfrak{l}}$.
Furthermore, since $\varpi\left(V_{x,y}\right)$ is contained in $V_{\varpi\left(x\right),\varpi\left(y\right)}^{\mathfrak{l}}$
by Lemma \ref{Lm:Vxyc} (iii) and since $\dim V_{\varpi\left(x\right),\varpi\left(y\right)}^{\mathfrak{l}}\leqslant\mathrm{b}_{\mathfrak{l}}$,
$V_{x,y}=V_{\varpi\left(x\right),\varpi\left(y\right)}^{\mathfrak{l}}+\mathfrak{p}_{\mathrm{u}}$.
\end{proof}
Set:
\[
\begin{array}{c}
\mathcal{R}_{+}':=\left\{ \alpha\in\mathcal{R}_{+}\mid\mathfrak{g}^{\alpha}\subset\mathfrak{p}_{\mathrm{u}}\right\}.
\end{array}
\]
Let $\beta_{1},\ldots,\beta_{l}$
be in $\Pi$, let $s_{i}$ be the reflexion associated to $\beta_{i}$
for all $i\in\left\{ 1,\ldots,l\right\} $ and let $I$ be the set
of $i\in\left\{ 1,\ldots,l\right\} $ such that $\beta_{i}\in\mathcal{R}'_{+}$.
\begin{lem}
\label{lem:winW_l}Let $w$ be in $W$ and let $s_{1}\ldots s_{p}$
be a reduced decomposition of $w$.

(i) If $w\left(\mathcal{R}_{+}'\right)\subseteq\mathcal{R}{}_{+}$,
then $w\in W_{\mathfrak{l}}$.

(ii) If $w\left(\mathcal{R}_{+}'\right)\subseteq\mathcal{R}{}_{+}\cup\mathcal{R}_{\mathfrak{l}}$,
then $w\in W_{\mathfrak{l}}$.\end{lem}
\begin{proof}
(i) We proceed by induction on the lenght of $w$ denoted by $\mathit{l}\left(w\right)$.
For $\mathit{l}\left(w\right)=1$, $w=s_{j}$ for some $j\in\left\{ 1,\ldots,l\right\} $.
If $j\in I$, then $s_{j}\left(\beta_{j}\right)=-\beta_{j}$ but $w\left(\mathcal{R}_{+}'\right)\subseteq\mathcal{R}{}_{+}$.
Hence $j\notin I$ and $w\in W_{\mathfrak{l}}$.
Suppose the property true for $\mathit{l}\left(w\right)\leqslant p-1$. If $p\in I$,
\[
s_{1}\ldots s_{p-1}\left(\beta_{p}\right)\in-\mathcal{R}_{+}
\]
since $s_{p}\left(\beta_{p}\right)=-\beta_{p}$. By \cite{key-14} Lemma 18.8.3, there exists $q\in\left\{ 1,\ldots,p-1\right\} $
such that 
\[
s_{1}\ldots s_{p}=s_{1}\ldots s_{q-1}s_{q+1}\ldots s_{p-1},
\]
so that $\mathit{l}\left(w\right)\neq p$. Hence $p\notin I$. Then $s_{p}\left(\mathcal{R}_{+}'\right)=\mathcal{R}_{+}'$ and 
\[
s_{1}\ldots s_{p-1}\left(\mathcal{R}_{+}'\right)=w\left(\mathcal{R}_{+}'\right)\subseteq\mathcal{R}_{+}.
\]
By induction hypothesis, 
\[
s_{1}\ldots s_{p-1}\in W_{\mathfrak{l}}.
\]
Hence $w\in W_{\mathfrak{l}}$ since $s_{p}\in W_{\mathfrak{l}}$,
whence the assertion.

(ii) Let $g_{w}$ be in $N_{G}\left(\mathfrak{h}\right)$ a representative
of $w$. Since $g_{w}\left(\mathfrak{p}_{\mathrm{u}}\oplus\mathfrak{h}\right)$
is contained in $\mathfrak{p}$, it is contained in $\mathfrak{b}'$
a Borel subalgebra contained in $\mathfrak{p}$, containing $\mathfrak{h}$.
Then there exists $w'\in W_{\mathfrak{l}}$, such that $g_{w'}\left(\mathfrak{b}'\right)=\mathfrak{b}$,
for some representative $g_{w'}$ of $w'$ in $N_{G}\left(\mathfrak{h}\right)$.
So $g_{w'}g_{w}\left(\mathfrak{p}_{\mathrm{u}}\oplus\mathfrak{h}\right)$
is contained in $\mathfrak{b}$ and $w'w\left(\mathcal{R}_{+}'\right)\subseteq\mathcal{R}_{+}$.
Hence, by (i), $w'w\in W_{\mathfrak{l}}$, whence $w\in W_{\mathfrak{l}}$
since $w'\in W_{\mathfrak{l}}$.\end{proof}
\begin{prop}
\label{prop-Parab}Let $\mathfrak{p}$ and $\mathfrak{p}'$ be two
parabolic subalgebras of $\mathfrak{g}$ such that $\mathfrak{p}_{\mathrm{u}}$
is contained in $\mathfrak{p'}$. If $\mathfrak{p}$ and $\mathfrak{p}'$ are
conjugate under $G$, then $\mathfrak{p}$=$\mathfrak{p}'$. \end{prop}
\begin{proof}
We suppose that $\mathfrak{p}'=g(\mathfrak{p})$, for some $g$ in
$G$. By Bruhat decomposition there exist $u$, $b\in\mathbf{B}$
and $w\in W$ such that $g=ug_{w}b$, where $g_{w}$ is a representative
of $w$ in $N_{G}\left(\mathfrak{h}\right)$. Since $g(\mathfrak{p})$
contains $\mathfrak{p}_{\mathrm{u}}$ and since $u^{-1}\left(\mathfrak{p}_{\mathrm{u}}\right)=\mathfrak{p}_{\mathrm{u}}$,
$ug_{w}\left(\mathfrak{p}\right)$ and $g_{w}\left(\mathfrak{p}\right)$
contains $\mathfrak{p}_{\mathrm{u}}$. Hence $\mathfrak{p}$ contains
$g_{w}^{-1}\left(\mathfrak{p}_{\mathrm{u}}\right)$. Then $w^{-1}\left(\mathcal{R}_{+}'\right)\subseteq\mathcal{R}{}_{+}\cup\mathcal{R}_{\mathfrak{l}}$.
Hence, by Lemma \ref{lem:winW_l} (ii), $w^{-1}\in W_{\mathfrak{l}}$,
whence $g_{w}\in\mathbf{L}$ and $g\in\mathbf{P}$. 
\end{proof}

\subsection{A desingularization of $\mathcal{P}^{\left(k\right)}$}

We consider the action of $\mathbf{P}$ on $G\times\mathfrak{p}^{k}$
given by 
\[
p.(g,x_{1},\ldots,x_{k})=\left(gp^{-1},p\left(x_{1}\right),\ldots,p\left(x_{k}\right)\right).
\]
 Let $\mu'_{k}$ be the morphism from $G\times\mathfrak{p}^{k}$ to
$\mathfrak{g}^{k}$ defined by 
\[
\mu'_{k}(g,x_{1},\ldots,x_{k})=\left(g\left(x_{1}\right),\ldots,g\left(x_{k}\right)\right)
\]
 and let $\mu_{k}$ be the morphism from $G\times_{\mathbf{P}}\mathfrak{p}^{k}$
to $\mathfrak{g}^{k}$ defined through the quotient by $\mu'_{k}$.
Let 
\[
\mathcal{P}^{\left(k\right)}:=\left\{ \left(x_{1},\ldots,x_{k}\right)\in\mathfrak{g}^{k}\mid\exists g\in G\textrm{ such that }\left(g\left(x_{1}\right),\ldots,g\left(x_{k}\right)\right)\in\mathfrak{p}^{k}\right\} .
\]
 Then $\mathcal{P}^{\left(k\right)}$ is the image of $G\times_{\mathbf{P}}\mathfrak{p}^{k}$
by $\mu_{k}$.
\begin{prop}
Let $k\geqslant2$, the variety $G\times_{\mathbf{P}}\mathfrak{p}^{k}$ is a desingularization of $\mathcal{P}^{\left(k\right)}$ and $\mu_{k}$ is the desingularization morphism.
The subvariety $\mathcal{P}^{\left(k\right)}$ of $\mathfrak{g}^k$ is closed, but it is
not normal.\end{prop}
\begin{proof}
According to Lemma \ref{Lem1.4}, $\mu_{k}$ is a projective morphism and $\mathcal{P}^{\left(k\right)}$
is closed in $\mathfrak{g}^k$ since $\mathcal{P}^{\left(k\right)}$ is the image of $G\times\mathfrak{p}^{k}$
by $\mu'_{k}$. 

Since $\Omega_{\mathfrak{g}}$ is an open subset of $\mathfrak{g}\times\mathfrak{g}$,
$\Omega_{\mathfrak{g}}\times\mathfrak{g}^{k-2}\cap\mathcal{P}^{\left(k\right)}$
is an open subset of $\mathcal{P}^{\left(k\right)}$. Let $\left(x_{1},\ldots,x_{k}\right)$
be in $\Omega_{\mathfrak{g}}\times\mathfrak{g}^{k-2}\cap\mathfrak{p}^{\left(k\right)}$.
Suppose that $\left(x_{1},\ldots,x_{k}\right)\in\left(g\left(\mathfrak{p}\right)\right)^k$ for some $g\in G$.
By Corollary \ref{cor:Vxy=00003Db}, $V_{x_{1},x_{2}}$ is contained
in $\mathfrak{p}$ and $g\left(\mathfrak{p}\right)$ and by Corollary \ref{cor:Vxy=00003DVl+pu},
$V_{x_{1},x_{2}}$ contains $\mathfrak{p}_{\mathrm{u}}$ and $g\left(\mathfrak{p}_{\mathrm{u}}\right)$.
Hence $\mathfrak{p}_{\mathrm{u}}$ is contained in $g\left(\mathfrak{p}\right)$. Then, by Proposition
\ref{prop-Parab}, $g\left(\mathfrak{p}\right)=\mathfrak{p}$. Hence, for all $\left(x_{1},\ldots,x_{k}\right)\in\Omega_{\mathfrak{g}}\times\mathfrak{g}^{k-2}\cap\mathcal{P}_{\mathfrak{}}^{\left(k\right)}$,
$\mid\mu_{k}^{-1}\left(x_{1},\ldots,x_{k}\right)\mid=1$, Whence $\mu_{k}$
is a birational morphism.

Let $x$ be in $\mathfrak{h}_{\mathrm{reg}}$ and let $g$ be in $G$. Then $$g\left(x\right)\in\mathfrak{p}\Leftrightarrow g^{-1}\left(\mathfrak{h}\right)\subset\mathfrak{p}\Leftrightarrow g\in N_{G}\left(\mathfrak{h}\right)\mathbf{P}.$$ According to Lemma \ref{lem:winW_l}, (ii), $N_{G}\left(\mathfrak{h}\right)\mathbf{P}/\mathbf{P}=W/W_\mathfrak{l}$ so that $\mid\mu_{k}^{-1}\left(x,0,\ldots,0\right)\mid>1$. Since $\mu_{k}$ is proper and birational, if $\mathcal{P}^{\left(k\right)}$ would be normal, then by Zariski's main Theorem \cite{key-5}, a finite fiber of $\mu_{k}$ would have cardinality $1$. So $\mathcal{P}^{\left(k\right)}$ is not normal.
\end{proof}
\subsection{ }

Let $\mathfrak{l}//\mathbf{L}$ be the categorical quotient of $\mathfrak{l}$ by $\mathbf{L}$. Let
$\chi_{0}$ be the closed subvariety of $\mathfrak{g}\times\mathfrak{h}$
such that $\Bbbk\left[\chi_{0}\right]=S\left(\mathfrak{g}\right)\otimes_{S\left(\mathfrak{h}\right)^{W}}S\left(\mathfrak{h}\right)$
and let $\chi$ be equal to $\chi_{0}//W_{\mathfrak{l}}$. For $x\in\mathfrak{p}$,
let $\widetilde{x}\in\mathfrak{l}$ such that $x-\widetilde{x}\in\mathfrak{p}_{\mathrm{u}}$
and let $\bar{x}$ be the image of $\widetilde{x}$ in $\mathfrak{l}//\mathbf{L}$.
\begin{lem}
\label{lem:tetta-isom}(i) Let $x$ and $x'$ be in $\mathfrak{p}_{\mathrm{reg}}$.
If $\left(x',\bar{x'}\right)$ is in $G.\left(x,\bar{x}\right)$,
then $x'$ is in $\mathbf{P}\left(x\right)$.

(ii) For $x$ in $\mathfrak{p}_{\mathrm{reg}}$, $G^{x}$ is contained
in $\mathbf{P}$.

(iii) The map 
\[
\begin{array}{ccc}
\theta:\mathbf{P}_{-,\mathrm{u}}\times\mathfrak{p}_{\mathrm{reg}} & \longrightarrow & \chi\\
\left(g,x\right) & \longmapsto & \left(g\left(x\right),\bar{x}\right)
\end{array}
\]
is an isomorphism onto an open subset of $\chi$.\end{lem}
\begin{proof}
(i) Let $x$ and $x'$ be in $\mathfrak{p}_{\mathrm{reg}}$ such that
$\left(x',\bar{x'}\right)\in G.\left(x,\bar{x}\right)$ and let $x=x_{s}+x_{n}$
and $x'=x'_{s}+x'_{n}$ be the Jordan decomposition of $x$ and $x'$
respectively. Since $\bar{x}=\bar{x'}$, there exists $l$ in $\mathbf{L}$
such that $\tilde{x}_{s}=l\left(\tilde{x'}_{s}\right)$, where
$\tilde{x}_{s}$ and $\tilde{x'}_{s}$ are the semisimple components
of $\tilde{x}$ and $\tilde{x'}$ respectively. Then we can suppose
that $\tilde{x}_{s}=\tilde{x'}_{s}$. The semisimple components $x_{s}$
and $x'_{s}$ of $x$ and $x'$ are conjugate under $\mathbf{P}$
since they are conjugate to $\tilde{x}_{s}$ under $\mathbf{P}$.
Then there exist $p$ and $p'$ in $\mathbf{P}$ such that $p\left(x_{s}\right)=p'\left(x'_{s}\right)\in\mathfrak{h}$.
Since $p\left(x_{s}\right)$ is a semisimple element, $\mathfrak{g}^{p\left(x_{s}\right)}$
is a reductive algebra. Then $\mathfrak{p}\cap\mathfrak{g}^{p\left(x_{s}\right)}$
is a parabolic subalgebra of $\mathfrak{g}^{p\left(x_{s}\right)}$,
since it contains the Borel subalgebra $\mathfrak{b}\cap\mathfrak{g}^{p\left(x_{s}\right)}$
of $\mathfrak{g}^{p\left(x_{s}\right)}$. Then $p\left(x_{n}\right)$
and $p'\left(x'_{n}\right)$ are regular nilpotent elements in $\mathfrak{p}\cap\mathfrak{g}^{p\left(x_{s}\right)}$,
hence they are conjugate under the parabolic subgroup $\mathbf{Q}$
of $G^{p\left(x_{s}\right)}$ of Lie algebra $\mathrm{ad}\left(\mathfrak{p}\cap\mathfrak{g}^{p\left(x_{s}\right)}\right)$.
Hence $p\left(x\right)$ and $p'\left(x'\right)$ are conjugate under $\mathbf{P}$.

(ii) By \cite{key-4} Proposition 14, $G^{x}$ is connex since $x$ is regular . Moreover
$\mathfrak{g}^{x}$ is contained in $\mathfrak{p}$ since $x$ is
in $\mathfrak{p}$. Hence $G^{x}$ is contained in $\mathbf{P}$.

(iii) Let $\left(g,x\right)$ and $\left(g',x'\right)$ be in $\mathbf{P}_{-,\mathrm{u}}\times\mathfrak{p}_{\mathrm{reg}}$
such that $\theta\left(g,x\right)=\theta\left(g',x'\right)$. By (i),
there exists $p\in\mathbf{P}$ such that $x'=p\left(x\right)$, then
$g^{-1}g'p$ is in $G^{x}$. Then, by (ii), $g^{-1}g'$ is in $\mathbf{P}_{-,\mathrm{u}}\cap\mathbf{P}=\left\{ 1_{\mathfrak{g}}\right\} $.
Whence $\left(g,x\right)=\left(g',x'\right)$. Since $\chi$ is the
categorical quotient of the variety $\chi_{0}$ by the finite subset
$W_{\mathfrak{l}}$ and since $\chi_{0}$ is normal, by \cite{key-1}
Lemma 3.1, $\chi$ is a normal variety. Then, by Zariski's Main Theorem
\cite{key-5}, $\theta$ is an isomorphism from $\mathbf{P}_{-,\mathrm{u}}\times\mathfrak{p}_{\mathrm{reg}}$ onto an open subset of
$\chi$.\end{proof}

\begin{cor}
There exists a birational morphism $\vartheta$ from $G\times_{\mathbf{P}}\mathfrak{p}$
to $\chi$ such that $\mu_1$ is the compound of $\vartheta$ and
the canonical projection from $\chi$ to $\mathfrak{g}$.\end{cor}
\begin{proof}
Let $\vartheta'$ be the morphism from $G\times\mathfrak{p}$ to $\chi$
defined by $\vartheta'\left(g,x\right)=\left(g\left(x\right),\bar{x}\right)$
and let $\vartheta$ be the morphism from $G\times_{\mathbf{P}}\mathfrak{p}$
to $\chi$ defined by $\vartheta'$ through the quotient. Let $\left(x,\bar{x}\right)$ be in $\chi$ such that $x$ in $\mathfrak{p}_{\mathrm{reg}}$ and let $\left(g_1,x_1\right)$ be in $G\times\mathfrak{p}$ such that\linebreak $\vartheta'\left(g_1,x_1\right)=\left(x,\bar{x}\right).$ Then $\left(g_1\left(x_1\right),\overline{x_1}\right)=\left(x,\bar{x}\right)$ and $\left(x_1,\overline{x_1}\right)\in G.\left(x,\bar{x}\right).$ Hence, by Lemma \ref{lem:tetta-isom} (i), there exists $p\in\mathbf{P}$ such that $x_1=p\left(x\right)$. So $g_1p\left(x\right)=x$ and $g_1p\in G^x$. By Lemma \ref{lem:tetta-isom} (ii), $g_1\in\mathbf{P}$ since $x\in\mathfrak{p}_{\mathrm{reg}}$. Hence the restriction
of $\vartheta$ to $G\times_{\mathbf{P}}\mathfrak{p}_{\mathrm{reg}}$
is injective, whence $\vartheta$ is birational since it is surjective.
\end{proof}
\subsection{}

Let

\[
{\sigma'_{k}:G\times\mathfrak{p}^{k}\rightarrow\chi^{k}\atop \left(g,x_{1},\ldots,x_{k}\right)\mapsto\left(g\left(x_{1}\right),\ldots,g\left(x_{k}\right),\overline{x_{1}},\ldots,\overline{x_{k}}\right),}
\]
where $\bar{x_{i}}$ is the image of $\widetilde{x_{i}}$ in $\mathfrak{l}//\mathbf{L}$.
Denote by $\sigma_{k}$ the morphism from $G\times_{\mathbf{P}}\mathfrak{p}^{k}$
to $\chi^k$ defined through the quotient by $\sigma'_{k}$. Let $\chi^{\left(k\right)}$ be the image of $G\times_{\mathbf{P}}\mathfrak{p}^{k}$ by $\sigma_{k}$
and let $\mathrm{pr}_{1,k}$ be the canonical projection from $\chi^{k}$
onto $\mathcal{P}^{\left(k\right)}$. 
\begin{prop}
\label{prop:desing-of-qik}(i) The varity $G\times_{\mathbf{P}}\mathfrak{p}^{k}$ is a desingularization
of $\chi^{\left(k\right)}$ of morphism $\sigma_{k}$.

(ii)For $x=\left(x_1,\ldots,x_k,y_1,\ldots,y_k\right)\in\chi^{\left(k\right)}$ such that $x_i\in\mathfrak{g}_{\mathrm{reg}}$ for some $i$ ,
the fiber of $\sigma_{k}$ at $x$ has one element.\end{prop}
\begin{proof}
(i) Since $\mu_{k}=\mathrm{pr}_{1,k}\circ\sigma_{k}$ and since $\mu_{k}$
is a projective and birational morphism, $\sigma_{k}$ is projective and birational.
Then $\sigma_{k}$ is a desingularization of $\chi^{\left(k\right)}$ since $G\times_{\mathbf{P}}\mathfrak{p}^{k}$
is a smooth variety.

(ii) Let $x$ be in $\chi^{\left(k\right)}$ such that the first component of $\mathrm{pr}_{1,k}\left(x\right)$ is in $\mathfrak{g}_{\mathrm{reg}}$.
Let $\overline{\left(g,x_1,\ldots,x_k\right)}$ and $\overline{\left(g',x'_1,\ldots,x'_k\right)}$
be in $\sigma_{k}^{-1}\left(x\right)$ and let $\left(g,x_{1},\ldots,x_{k}\right)$ and $\left(g',x'_{1},\ldots,x'_{k}\right)$
be in $G\times\mathfrak{p}^{k}$ representatives of $\overline{\left(g,x_{1},\ldots,x_{k}\right)}$
and $\overline{\left(g',x'_{1},\ldots,x'_{k}\right)}$ respectively.
So 
\[
\left(g\left(x_{1}\right),\overline{x_{1}}\right)=\left(g'\left(x'_{1}\right),\overline{x'_{1}}\right),
\]
then 
\[
\left(x_{1},\overline{x_{1}}\right)\in G.\left(x'_{1},\overline{x'_{1}}\right).
\]
Hence, by Lemma \ref{lem:tetta-isom} (i), $x'_{1}\in\mathbf{P}\left(x_{1}\right)$.
Then there exists $p\in\mathbf{P}$, such that 
\[
g\left(x_{1}\right)=g'p\left(x_{1}\right).
\]
So $$g^{-1}g'p\in G^{x_{1}},$$ then by Lemma \ref{lem:tetta-isom} (ii), $$g^{-1}g'p\in\mathbf{P}$$
since $x_{1}\in\mathfrak{p}_{\mathrm{reg}}$. Hence $g^{-1}g'\in\mathbf{P}$
and $g\left(\mathfrak{p}\right)=g'\left(\mathfrak{p}\right)$, whence
the assertion since $\sigma_{k}$ is $\mathfrak{S}_{k}$-equivariant.\end{proof}
\begin{cor}
The variety $\chi^{\left(k\right)}$ is an irreducible component of
the inverse image of $\mathcal{P}^{\left(k\right)}$ in $\chi^{\left(k\right)}$.\end{cor}
\begin{proof}
Let $\pi$ be the canonical projection from $\chi^{k}$ to $\mathfrak{g}^{k}$. Since $\chi_0\rightarrow\mathfrak{g}$ is finite, $\chi\rightarrow\mathfrak{g}$ and $\pi$ are finite morphism. So $\pi^{-1}\left(\mathcal{P}^{\left(k\right)}\right)$,
$\chi^{\left(k\right)}$ and $\mathcal{P}^{\left(k\right)}$ have
the same dimension, whence the corollary since $\chi^{\left(k\right)}$
is irreducible as an image of an irreducible variety.
\end{proof}

\begin{prop} Let $\Phi$ be the morphism from $\Bbbk[\mathfrak{l}^k]^{\mathbf{L}^k}$ to $\Bbbk[\chi^{\left(k\right)}]$ defined by $$\Phi\left(P\right)\left(x_1,\ldots,x_k,\overline{x_{1}},\ldots,\overline{x_{k}}\right)=P\left(\overline{x_{1}},\ldots,\overline{x_{k}}\right),$$ where $P\in\Bbbk[\mathfrak{l}^k]^{\mathbf{L}^k}$ and  $\left(x_1,\ldots,x_k,\overline{x_{1}},\ldots,\overline{x_{k}}\right)\in\chi^{\left(k\right)}$. Then $\Phi$ is an isomorphism from $\Bbbk[\mathfrak{l}^k]^{\mathbf{L}^k}$ onto $\Bbbk[\chi^{\left(k\right)}]^G$.
\end{prop}
\begin{proof} Clearly the image of $\Phi$ is contained in $\Bbbk[\chi^{\left(k\right)}]^G$ and $\Phi$ is injective. Let $Y_k$ be equal $\sigma'_k\left(\{1\}\times\mathfrak{p}^k\right)$. Let $P$ be in $\Bbbk[\chi^{\left(k\right)}]^G$ and let $\overline{P}$ be in $\Bbbk[\mathfrak{l}^k]^{\mathbf{L}^k}$ defined by $$\overline{P}\left(\overline{x_{1}},\ldots,\overline{x_{k}}\right)=P\left(x_1,\ldots,x_k,\overline{x_{1}},\ldots,\overline{x_{k}}\right),$$ where $\left(x_1,\ldots,x_k,\overline{x_{1}},\ldots,\overline{x_{k}}\right)\in\chi^{\left(k\right)}$. So $$\left(P-\Phi\left(\overline{P}\right)\right)_{|Y_k}=0.$$ Hence $$P=\Phi\left(\overline{P}\right)$$ since $\chi^{\left(k\right)}=G.Y_k$, whence $\Phi$ is bijective.
\end{proof}
\subsection{}
We recall that $\theta$ is the map 
\[
\begin{array}{ccc}
\mathbf{P}_{-,\mathrm{u}}\times\mathfrak{p}_{\mathrm{reg}} & \longrightarrow & \chi\\
\left(g,x\right) & \longmapsto & \left(g\left(x\right),\bar{x}\right).
\end{array}
\]

\begin{prop}
\label{prop:Wk-smooth}Let 
\[
W_{k}:=\left\{ \left(x_{1},\ldots,x_{k},\overline{y_{1}},\ldots,\overline{y_{k}}\right)\in\chi^{\left(k\right)}\mid\exists i\in\left\{ 1,\ldots,n\right\} \mbox{ such that }x_{i}\in\mathfrak{g}_{\mathrm{reg}}\right\} .
\]

(i) The subset $W_{k}$ is a smooth open subset of $\chi^{\left(k\right)}$.

(ii) The codimensions of $\chi^{\left(k\right)}\setminus W_{k}$ in
$\chi^{\left(k\right)}$ and of $G\times_{\mathbf{P}}\left(\mathfrak{p}^{k}\setminus\mathrm{pr}_{1,k}\left(W_{k}\right)\right)$
in $G\times_{\mathbf{P}}\mathfrak{p}^{k}$ are at least $k$.

(iii) The restriction of $\sigma_{k}$ to $\sigma_{k}^{-1}\left(W_{k}\right)$
is an isomorphism onto $W_{k}$.\end{prop}
\begin{proof}
(i) Let $W'_{k}$ be the inverse image of $\theta\left(\mathbf{P}_{-,\mathrm{u}}\times\mathfrak{p}_{\mathrm{reg}}\right)$
by the projection 
\[
\begin{array}{ccc}
\chi^{\left(k\right)} & \longrightarrow & \chi\\
\left(x_{1},\ldots,x_{k},\overline{y_{1}},\ldots,\overline{y_{k}}\right) & \longmapsto & \left(x_{1},\overline{y_{1}}\right).
\end{array}
\]
 By Lemma \ref{lem:tetta-isom} (iii), the image of $\theta$ is an
open set of $\chi$. Then $W'_{k}$ is an open subset of $\chi^{\left(k\right)}$.

Let $\left(x_{1},\ldots,x_{k}\right)$ be in $\mathfrak{p}^{k}$ and
let $g$ be in $G$ such that $\left(g\left(x_{1}\right),\ldots,g\left(x_{k}\right),\overline{x_{1}},\ldots,\overline{x_{k}}\right)$
is in $W'_{k}$. Then $x_{1}$ is in $\mathfrak{p}_{\mathrm{reg}}$
and for some $g'$ in $\mathbf{P}_{-,\mathrm{u}}$ and for some $x'_{1}$
in $\mathfrak{p}_{\mathrm{reg}}$ 
\[
g.\left(x_{1},\overline{x_{1}}\right)=g'.\left(x'_{1},\overline{x'_{1}}\right).
\]
 By Lemma \ref{lem:tetta-isom} (i), there exists $p\in\mathbf{P}$
such that $x'_{1}=p\left(x_{1}\right)$. So $g^{-1}g'p$ is in $G^{x_{1}}$,
since $G^{x_{1}}\subset\mathbf{P}$ by Lemma \ref{lem:tetta-isom}
(ii). Hence $g\in\mathbf{P}_{-,\mathrm{u}}\mathbf{P}$. As a result,
the map
\[
\begin{array}{ccc}
\mathbf{P}_{-,\mathrm{u}}\times\mathfrak{p}_{\mathrm{reg}}\times\mathfrak{p}^{k-1} & \longrightarrow & W'_{k}\\
\left(g,x_{1},\ldots,x_{k}\right) & \longmapsto & \left(g\left(x_{1}\right),\ldots,g\left(x_{k}\right),\overline{x_{1}},\ldots,\overline{x_{k}}\right)
\end{array}
\]
 is an isomorphism whose inverse is 
\[
\begin{array}{ccc}
W'_{k} & \longrightarrow & \mathbf{P}_{-,\mathrm{u}}\times\mathfrak{p}_{\mathrm{reg}}\times\mathfrak{p}^{k-1}\\
\left(x_{1},\ldots,x_{k},\overline{y_{1}},\ldots,\overline{y_{k}}\right) & \longmapsto & \left(q,z,q^{-1}\left(x_{2}\right),\ldots,q^{-1}\left(x_{k}\right)\right),
\end{array}
\]
 with $q$ and $z$ are the first and the second components of $\theta^{-1}\left(x_{1},\overline{y_{1}}\right)$.
As a result, $W'_{k}$ and $G.W'_{k}$ are smooth open subsets of
$\chi^{\left(k\right)}$, whence the assertion since $\chi^{\left(k\right)}$
is $\mathfrak{S}_{k}$-invariant.

(ii) By definition, $G\times_{\mathbf{P}}\left(\mathfrak{p}^{k}\setminus\mathrm{pr}_{1,k}\left(W_{k}\right)\right)$
is contained in $G\times_{\mathbf{P}}\left(\mathfrak{p\setminus\mathfrak{p}_{\mathrm{reg}}}\right)^{k}$
and $\chi^{\left(k\right)}\setminus W_{k}$ is contained in the image
of $G\times_{\mathbf{P}}\left(\mathfrak{\mathfrak{p\setminus\mathfrak{p}_{\mathrm{reg}}}}\right)^{k}$
by $\sigma_{k}$. Then 
\[
{\dim\chi^{\left(k\right)}\setminus W_{k}\leqslant\dim G\times_{\mathbf{P}}\left(\mathfrak{p\setminus\mathfrak{p}_{\mathrm{reg}}}\right)^{k}\leqslant k\dim\mathfrak{p}+\dim\mathfrak{p}_{\mathrm{u}}-k\atop \dim G\times_{\mathbf{P}}\left(\mathfrak{p}^{k}\setminus\mathrm{pr}_{1,k}\left(W_{k}\right)\right)\leqslant\dim G\times_{\mathbf{P}}\left(\mathfrak{\mathfrak{p\setminus\mathfrak{p}_{\mathrm{reg}}}}\right)^{k}\leqslant k\dim\mathfrak{p}+\dim\mathfrak{p}_{\mathrm{u}}-k,}
\]
 whence the assertion since $\dim G\times_{\mathbf{P}}\mathfrak{p}^{k}=\dim\mathcal{P}^{\left(k\right)}=k\dim\mathfrak{p}+\dim\mathfrak{p}_{\mathrm{u}}$
.

(iii) By Proposition \ref{prop:desing-of-qik} (ii), the restriction
of $\sigma_{k}$ to $\sigma_{k}^{-1}\left(W_{k}\right)$ is bijective.
Whence the assertion by Zariski Main Theorem \cite{key-5} since $W_{k}$
is a smooth open subset of $\chi^{\left(k\right)}$ by (i).\end{proof}

\section{On the variety $\mathcal{P}_{\mathfrak{\mathrm{u}}}^{\left(k\right)}$}

Recall that $\mu_{k}$ is the morphism from $G\times_{\mathbf{P}}\mathfrak{p}^{k}$
to $\mathfrak{g}$ defined by the map $$(g,x_{1},\ldots,x_{k})\mapsto\left(g\left(x_{1}\right),\ldots,g\left(x_{k}\right)\right)$$
from $G\times\mathfrak{p}^{k}$ to $\mathfrak{g}^{k}$ through the
quotient. Let $\tau_{k}$ be the restriction of $\mu_{k}$ to $G\times_{\mathbf{P}}\mathfrak{p}_{\mathrm{u}}^{k}$
and let 
\[
\mathcal{P}_{\mathfrak{\mathrm{u}}}^{\left(k\right)}:=\left\{ \left(x_{1},\ldots,x_{k}\right)\in\mathfrak{g}^{k}\mid\exists g\in G\textrm{ such that }\left(g\left(x_{1}\right),\ldots,g\left(x_{k}\right)\right)\in\mathfrak{p}_{\mathrm{u}}^{k}\right\} .
\]
\subsection{}

Let $\Bbbk\left(G\times_{\mathbf{P}}\mathfrak{p}_{\mathrm{u}}\right)$
and $\Bbbk\left(G\left(\mathfrak{p}_{\mathrm{u}}\right)\right)$ be
the fields of the rational functions of $G\times_{\mathbf{P}}\mathfrak{p}_{\mathrm{u}}$
and $G\left(\mathfrak{p}_{\mathrm{u}}\right)$ respectively and let
$\mathcal{A}$ be the integral closure of $\Bbbk\left[G\left(\mathfrak{p}_{\mathrm{u}}\right)\right]$
in $\Bbbk\left(G\times_{\mathbf{P}}\mathfrak{p}_{\mathrm{u}}\right)$.
\begin{lem}
The morphism $\tau_{1}$ is a projective morphism. Moreover, $\Bbbk\left(G\times_{\mathbf{P}}\mathfrak{p}_{\mathrm{u}}\right)$
is an algebraic extension of finite degree of $\Bbbk\left(G\left(\mathfrak{p}_{\mathrm{u}}\right)\right)$.\end{lem}
\begin{proof}
According to Lemma \ref{Lem1.4}, $\tau_{1}$ is a projective
morphism. Since $G\left(\mathfrak{p}_{\mathrm{u}}\right)$ is the
image of $G\times_{\mathbf{P}}\mathfrak{p}_{\mathrm{u}}$ by $\tau_{1}$,
$\tau_{1}$ induces an embedding from $\Bbbk\left(G\left(\mathfrak{p}_{\mathrm{u}}\right)\right)$
into $\Bbbk\left(G\times_{\mathbf{P}}\mathfrak{p}_{\mathrm{u}}\right)$.
Hence $\Bbbk\left(G\times_{\mathbf{P}}\mathfrak{p}_{\mathrm{u}}\right)$
is an algebraic extension of $\Bbbk\left(G\left(\mathfrak{p}_{\mathrm{u}}\right)\right)$
of finite degree since $\dim G\times_{\mathbf{P}}\mathfrak{p}_{\mathrm{u}}=\dim G\left(\mathfrak{p}_{\mathrm{u}}\right)$.
\end{proof}
Let $X=Spec\mathcal{A}$ and let $\alpha$ be the morphism from $X$
to $G\left(\mathfrak{p}_{\mathrm{u}}\right)$ such that its comorphism
is the canonical injection from $\Bbbk\left[G\left(\mathfrak{p}_{\mathrm{u}}\right)\right]$
into $\mathcal{A}$.
\begin{prop}
\label{prop:tauX}There exists a unique morphism $\tau_{X}$ from
$G\times_{\mathbf{P}}\mathfrak{p}_{\mathrm{u}}$ to $X$ such that
$\tau_{1}=\alpha\circ\tau_{X}$. Moreover, $G\times_{\mathbf{P}}\mathfrak{p}_{\mathrm{u}}$
is a desingularization of $X$ of morphism $\tau_{X}$.\end{prop}
\begin{proof}
The variety $G\times_{\mathbf{P}}\mathfrak{p}_{\mathrm{u}}$ is a
smooth variety since it is a vector bundle over the smooth variety
$G/\mathbf{P}$. Identify $\Bbbk\left(G\left(\mathfrak{p}_{\mathrm{u}}\right)\right)$
with a subfield of $\Bbbk\left(G\times_{\mathbf{P}}\mathfrak{p}_{\mathrm{u}}\right)$
by the comorphism of $\tau_{1}$. Let $U$ be an affine open subset
of $G\times_{\mathbf{P}}\mathfrak{p}_{\mathrm{u}}$. Then $\Bbbk\left[U\right]$
contains $\Bbbk\left[G\left(\mathfrak{p}_{\mathrm{u}}\right)\right]$
and the comorphism of the restriction map of $\tau_{1}$ to $U$ is
the canonical injection of $\Bbbk\left[G\left(\mathfrak{p}_{\mathrm{u}}\right)\right]$
in $\Bbbk\left[U\right]$. Since $U$ is a smooth variety, it is normal and $\Bbbk\left[U\right]$ is integrally closed in $\Bbbk\left(G\times_{\mathbf{P}}\mathfrak{p}_{\mathrm{u}}\right)$.
Hence $\Bbbk\left[U\right]$ contains $\mathcal{A}$. Then there exists
a morphism $\tau_{X,U}$ from $U$ to $X$ such that its comorphism
is the canonical injection of $\mathcal{A}$ in $\Bbbk\left[U\right]$.
For $U'$ an affine open subset of $G\times_{\mathbf{P}}\mathfrak{p}_{\mathrm{u}}$,
the restrictions of $\tau_{X,U}$ and $\tau_{X,U'}$ on $U\cap U'$
are equal since they have the same comorphism. Hence there exists
a morphism $\tau_{X}$ from $G\times_{\mathbf{P}}\mathfrak{p}_{\mathrm{u}}$
to $X$ which extends $\tau_{X,U}$ for all affine open subset $U$
and satisfies the equality $\tau_{1}=\alpha\circ\tau_{X}$. The morphism
$\tau_{X}$ is unique since its comorphism is the canonical injection
from $\mathcal{A}$ to $\Bbbk\left(G\times_{\mathbf{P}}\mathfrak{p}_{\mathrm{u}}\right)$.
Since $\Bbbk\left(G\times_{\mathbf{P}}\mathfrak{p}_{\mathrm{u}}\right)$
is the fraction field of $\mathcal{A}$, $\tau_{X}$ is birational.
Moreover, since $\tau_{1}$ is a projective morphism, $\tau_{X}$
is too, whence the proposition. 
\end{proof}
Denote by $\kappa$ the quotient map from $G\times\mathfrak{p}_{\mathrm{u}}$
to $G\times_{\mathbf{P}}\mathfrak{p}_{\mathrm{u}}$ and denote by
$\iota$ the morphism from $\left\{ 1\right\} \times\mathfrak{p}{}_{\mathrm{u}}$
to $X$ defined by $\iota\left(1,x\right)=\tau_{X}\circ\kappa\left(1,x\right)$.
Let $X'$ be the image of $\left\{ 1\right\} \times\mathfrak{p}'_{\mathrm{u}}$
by $\iota$ and let $\delta$ be the morphism from $\mathbf{P}_{-,\mathrm{u}}\times X'$
to $X$ defined by $\delta\left(p,x\right)=p.x$. 
\begin{prop}
\label{prop:stab-yota-x}(i) For $x\in\mathfrak{p}'_{\mathrm{u}}$,
the stabilizer of $\iota\left(1,x\right)$ in $G$ equals $\mathbf{P}\cap G^{x}$.

(ii) The morphism $\delta$ is an isomorphism onto an open subset
of $X$.

(iii) Let $g$ be in $G$ and let $x$ be in $X'$. If $g(x)$ is
in $X'$, then $g$ is in $\mathbf{P}$.\end{prop}
\begin{proof}
(i) Let $x$ be in $\mathfrak{p}'_{\mathrm{u}}$. Then, by \cite{key-16} Theorem 7.1.1, $G^x$ and $\mathbf{P}\cap G^x$ have the same identity component. Then $\mathbf{P}\cap G^x$ is a subgroup of finite index of $G^{\iota\left(1,x\right)}$ and $\Bbbk\left(G/\mathbf{P}\cap G^{x}\right)$ is an algebraic extension $\Bbbk\left(G/G^{\iota\left(1,x\right)}\right)$ of degree the index of $\mathbf{P}\cap G^x$ in $G^{\iota\left(1,x\right)}$. Moreover, $\Bbbk\left(G/G^{\iota\left(1,x\right)}\right)$ is an algebraic extension of finite degree of a transcendental extension of $\Bbbk$. Since $\tau_X$ is a birational morphism and since $G.\iota\left(1,x\right)$ is an open subset of $X$, $\Bbbk\left(G/G^{\iota\left(1,x\right)}\right)$ and $\Bbbk\left(G\times_{\mathbf{P}}\mathfrak{p}_{\mathrm{u}}\right)$ are isomorphic. Since $G\times_{\mathbf{P}}\mathbf{P}\left(x\right)$ is an open subset of $G\times_{\mathbf{P}}\mathfrak{p}_{\mathrm{u}}$, $\Bbbk\left(G\times_{\mathbf{P}}\mathfrak{p}_{\mathrm{u}}\right)$ and $\Bbbk\left(G/\mathbf{P}\cap G^{x}\right)$ are isomorphic. So that $\Bbbk\left(G/G^{\iota\left(1,x\right)}\right)$ and $\Bbbk\left(G/\mathbf{P}\cap G^{x}\right)$ are isomorphic. Hence $G^{\iota\left(1,x\right)}=\mathbf{P}\cap G^{x}$.

(ii) Let $\left(p_{1},x_{1}\right)$ and $\left(p_{2},x_{2}\right)$
be in $\mathbf{P}_{-,\mathrm{u}}\times X'$ such that $p_{1}.x_1=p_{2}.x_{2}$
and let $x_{1}'$ and $x_{2}'$ be in $\mathfrak{p}'{}_{\mathrm{u}}$
such that $\iota\left(1,x_{1}'\right)=x_{1}$ and $\iota\left(1,x_{2}'\right)=x_{2}$.
Let $q$ be in $\mathbf{P}$ such that $x_{2}'=q\left(x_{1}'\right)$.
Then $x_{2}=q.x_{1}$. So
\[
\begin{array}{ccccc}
 & p_{1}.x_{1} & = & p_{2}.x_{2}=p_{2}q.x_{1}\\
\Rightarrow & p_{1}^{-1}p_{2}q & \in & G^{x_{1}}=G^{\iota\left(1,x_{1}'\right)} & =\mathbf{P}\cap G^{x_{1}'},
\end{array}
\]
 by (i). So 
\[
\begin{array}{cccccc}
 & p_{1}^{-1}p_{2} & \in & \mathbf{P}\cap\mathbf{P}_{-,\mathrm{u}} & = & \left\{ 1_{\mathfrak{g}}\right\} \\
\Rightarrow & p_{1}=p_{2} & \Rightarrow & x_{1}=x_{2}.
\end{array}
\]
 Hence $\delta$ is an injective morphism. Since $X$ is a normal
variety by definition, by Zariski's Main Theorem \cite{key-5}, $\delta$
is an isomorphism onto an open subset of $X$.

(iii) Let $g\in G$ and let $x\in X'$ such that $g\left(x\right)\in X'$.
Then there exist 
\[
x_{1}\mbox{, }y_{1}\in\mathfrak{p}{}_{\mathrm{u}}'\mbox{ and }p\in\mathbf{P}
\]
 Such that 
\[
\begin{array}{cccc}
x=\iota\left(1,x_{1}\right), & g.x=\iota\left(1,y_{1}\right) & \mbox{and} & y_{1}=p\left(x_{1}\right).\end{array}
\]
 So
\[
g.\iota\left(1,x_{1}\right)=g.x=\iota\left(1,y_{1}\right)=p.\iota\left(1,x_{1}\right)
\]
and then $p^{-1}g\in G^{\iota\left(1,x_{1}\right)}$. Hence, by (i),
$g\in\mathbf{P}$.
\end{proof}
Let $\tau_{X^{k}}$ be the morphism from $G^{k}\times_{\mathbf{P}^{k}}\mathfrak{p}{}_{\mathrm{u}}^{k}$
to $X^{k}$ defined by 
\[
\tau_{X^{k}}\overline{\left(x_{1},\ldots,x_{k}\right)}=\left(\tau_{X}\left(x_{1}\right),\ldots,\tau_{X}\left(x_{k}\right)\right),
\]
let $\tau_{X^{\left(k\right)}}$ be the restriction of $\tau_{X^{k}}$
to $G\times_{\mathbf{P}}\mathfrak{p}{}_{\mathrm{u}}^{k}$ and let
$X^{\left(k\right)}$ be the image of $\tau_{X^{\left(k\right)}}$.
\begin{prop}
\label{prop:Vk-smooth}Let 
\[
V_{k}=\left\{ \left(x_{1},\ldots x_{k}\right)\in X^{\left(k\right)}\mid\exists g\in G\mbox{ and }i\in\left\{ 1,\ldots,k\right\} \mbox{ such that }g\left(x_{i}\right)\in\mathbf{P}_{-,\mathrm{u}}\left(X'\right)\right\} .
\]

(i) The subset $V_{k}$ is a smooth open subset of $X^{\left(k\right)}$.

(ii) There exists a canonical finite surjective morphism $\varphi$ from $X^{\left(k\right)}$
to $\mathcal{P}_{\mathrm{u}}^{\left(k\right)}$.

(iii) The codimension of $X^{\left(k\right)}\backslash V_{k}$ in
$X^{\left(k\right)}$ and the codimension of $G\times_{\mathbf{P}}\left(\mathfrak{p}{}_{\mathrm{u}}^{k}\setminus\varphi\left(V_{k}\right)\right)$
in $G\times_{\mathbf{P}}\mathfrak{p}{}_{\mathrm{u}}^{k}$ are at least
$k$.

(iv) The restriction of $\tau_{X^{\left(k\right)}}$ to $\tau_{X^{\left(k\right)}}^{-1}\left(V_{k}\right)$
is an isomorphism onto $V_{k}$.\end{prop}
\begin{proof}
(i) Set $V_{k}':=\mathbf{P}_{-,\mathrm{u}}\left(X'\right)\times X^{k-1}\cap X^{\left(k\right)}$
and
\[
V':=\left\{ \left(x_{1},\ldots,x_{k}\right)\in\left(\iota\left(\left\{ 1\right\} \times\mathfrak{p}{}_{\mathrm{u}}\right)\right)^{k}\mid\alpha\left(x_{1}\right)\in\mathfrak{p}{}_{\mathrm{u}}'\right\} .
\]
By Proposition $\ref{prop:stab-yota-x}$ (ii), $V_{k}'$ is an open
subset of $X^{\left(k\right)}$. Let $g\in G$ and $x_{1}\in X'$
such that $g.x_{1}\in\mathbf{P}_{-,\mathrm{u}}\left(X'\right)$.
Then there exist 
\[
\begin{array}{ccccc}
q\in\mathbf{P}_{-,\mathrm{u}}, & y_{1}\in X', & x_{1}',y_{1}'\in\mathfrak{p}{}_{\mathrm{u}}' & \mbox{and} & p\in\mathbf{P}\end{array}
\]
 such that
\[
\begin{array}{ccccc}
g\left(x_{1}\right)=q\left(y_{1}\right), & x_{1}=\iota\left(1,x_{1}'\right), & y_{1}=\iota\left(1,y_{1}'\right) & \mbox{and} & y_{1}'=p\left(x_{1}'\right)\end{array}.
\]
 So
\[
g.\iota\left(1,x_{1}'\right)=g.x_{1}=q.\iota\left(1,y_{1}'\right)=qp.\iota\left(1,x_{1}'\right)
\]
 and $p^{-1}q^{-1}g\in G^{\iota\left(1,x_{1}'\right)}$. Hence, by
Proposition $\ref{prop:stab-yota-x}$ (i), $g\in\mathbf{P}_{-,\mathrm{u}}\mathbf{P}$.
As a result 
\[
\begin{array}{ccc}
\mathbf{P}_{-,\mathrm{u}}\times V' & \longrightarrow & V_{k}'\\
\left(g,x_{1},\ldots,x_{k}\right) & \longmapsto & \left(g.x_{1},\ldots,g.x_{k}\right)
\end{array}
\]
 is an isomorphism whose inverse is given by 
\[
\begin{array}{ccc}
V_{k}' & \longrightarrow & \mathbf{P}_{-,\mathrm{u}}\times V'\\
\left(x_{1},\ldots,x_{k}\right) & \longmapsto & \left(g_{1},y_{1},g_{1}^{-1}.x_{2},\ldots,g_{1}^{-1}.x_{k}\right),
\end{array}
\]
where $\left(g_{1},y_{1}\right)=\delta^{-1}\left(x_{1}\right)$. Hence
$V_{k}'$ and $G.V_{k}'$ are smooth open subsets of $X^{\left(k\right)}$,
whence the assertion since $X^{\left(k\right)}$ is $\mathfrak{S}_{k}$-invariant.

(ii) The canonical projection $\alpha$ from $X$ to $G\left(\mathfrak{p}{}_{\mathrm{u}}\right)$
induces a morphism $\varphi$ from $X^{\left(k\right)}$ to $\mathcal{P}_{\mathrm{u}}^{\left(k\right)}$
such that $\tau_{k}=\varphi\circ\tau_{X^{\left(k\right)}}$. Since
$\alpha$ is a finite morphism, $\varphi$ is too and since $\tau_{k}$
is surjective, $\varphi$ is too.

(iii) By definition of $V_{k}$, $X^{\left(k\right)}\backslash V_{k}$
is contained in the image of $G\times_{\mathbf{P}}\left(\mathfrak{p}{}_{\mathrm{u}}\setminus\mathfrak{p}{}_{\mathrm{u}}'\right)^{k}$
by $\tau_{X^{\left(k\right)}}$ and $G\times_{\mathbf{P}}\left(\mathfrak{p}{}_{\mathrm{u}}^{k}\setminus\varphi\left(V_{k}\right)\right)$
is contained in $G\times_{\mathbf{P}}\left(\mathfrak{p}{}_{\mathrm{u}}\setminus\mathfrak{p}{}_{\mathrm{u}}'\right)^{k}$.
Hence 
\[
{\dim X^{\left(k\right)}\setminus V_{k}\leqslant\dim G\times_{\mathbf{P}}\left(\mathfrak{p}{}_{\mathrm{u}}\setminus\mathfrak{p}{}_{\mathrm{u}}'\right)^{k}\leqslant\left(k+1\right)\dim\mathfrak{p}{}_{\mathrm{u}}-k\atop \dim G\times_{\mathbf{P}}\left(\mathfrak{p}^{k}\setminus\varphi\left(V_{k}\right)\right)\leqslant\dim G\times_{\mathbf{P}}\left(\mathfrak{p}{}_{\mathrm{u}}\setminus\mathfrak{p}{}_{\mathrm{u}}'\right)^{k}\leqslant\left(k+1\right)\dim\mathfrak{p}{}_{\mathrm{u}}-k,}
\]
 whence the assertion since $\dim G\times_{\mathbf{P}}\mathfrak{\mathfrak{p}{}_{\mathrm{u}}}=\left(k+1\right)\dim\mathfrak{p}{}_{\mathrm{u}}$.

(iv) Let $\left(1,x_{1},\ldots,x_{k}\right)$ and $\left(g,y_{1},\ldots,y_{k}\right)$
be in $G\times\mathfrak{p}{}_{\mathrm{u}}^{k}$ such that $x_{1}$
and $y_{1}$ are in $\mathfrak{p}{}_{\mathrm{u}}'$ and 
\[
\tau_{X^{\left(k\right)}}\overline{\left(1,x_{1},\ldots,x_{k}\right)}=\tau_{X^{\left(k\right)}}\overline{\left(g,y_{1},\ldots,y_{k}\right)}\in V_{k},
\]
 where $\overline{\left(1,x_{1},\ldots,x_{k}\right)}$ and $\overline{\left(g,y_{1},\ldots,y_{k}\right)}$
are the images of $\left(1,x_{1},\ldots,x_{k}\right)$ and $\left(g,y_{1},\ldots,y_{k}\right)$
in $G\times_{\mathbf{P}}\mathfrak{p}{}_{\mathrm{u}}^{k}$ by the quotient
map. Then there exists $p\in\mathbf{P}$ such that $y_{1}=p\left(x_{1}\right)$.
So 
\[
\tau_{X}\circ\kappa\left(1,x_{1}\right)=\tau_{X}\circ\kappa\left(g,y_{1}\right)=g.\tau_{X}\circ\kappa\left(1,y_{1}\right)=gp.\tau_{X}\circ\kappa\left(1,x_{1}\right),
\]
 then 
\[
\iota\left(1,x_{1}\right)=gp.\iota\left(1,x_{1}\right)\in X'.
\]
So, by Proposition \ref{prop:stab-yota-x} (iii), $g\in\mathbf{P}$
and $\overline{\left(1,x_{1},\ldots,x_{k}\right)}=\overline{\left(g,y_{1},\ldots,y_{k}\right)}$.
Hence the restriction of $\tau_{X^{\left(k\right)}}$ to $\tau_{X^{\left(k\right)}}^{-1}\left(V_{k}\right)$
is injective since $V_{k}$ is $\mathfrak{S}_{k}$-invariant, whence
the assertion by Zariski Main Theorem \cite{key-5} since $V_{k}$
is smooth open subset of $X^{\left(k\right)}$ by (i).\end{proof}
\begin{prop}
\label{prop:desing-Xk}The variety $G\times_{\mathbf{P}}\mathfrak{p}{}_{\mathrm{u}}^{k}$
is a desingularization of $X^{\left(k\right)}$ of morphism $\tau_{X^{\left(k\right)}}$.\end{prop}
\begin{proof}
Since $G\times_{\mathbf{P}}\mathfrak{p}^{k}{}_{\mathrm{u}}$ is a
vector bundle on the smooth variety $G/\mathbf{P}$, it is a smooth
variety. According to Lemma \ref{Lem1.4}, $\tau_{k}$ is a projective
morphism then $\tau_{X^{\left(k\right)}}$ is too since $\tau_{k}=\varphi\circ\tau_{X^{\left(k\right)}}$.
Hence, by Proposition \ref{prop:Vk-smooth} (iv), $G\times_{\mathbf{P}}\mathfrak{p}^{k}_{\mathrm{u}}$
is a desingularization of $X^{\left(k\right)}$.\end{proof}

\subsection{}

Let $x$ be in $\mathfrak{p}_{\mathrm{u}}'$ and let $\left(G^{x}\right)_{0}$
be the neutral component of $G^{x}$. The canonical map $G/G^x\rightarrow G/\left(G^x\right)_0$ induces an embedding of $\Bbbk\left[G\left(\mathfrak{p}_{\mathrm{u}}\right)\right]$ into $\Bbbk\left(G/\left(G^x\right)_0\right)$. Let $\mathcal{C}$ be the integral closure of $\Bbbk\left[G\left(\mathfrak{p}_{\mathrm{u}}\right)\right]$ in $\Bbbk\left(G/\left(G^{x}\right)_{0}\right)$.
\begin{prop}
(i) The field $\Bbbk\left(G/\left(G^{x}\right)_{0}\right)$ is a Galois
extension of the field $\Bbbk\left(G/G^{x}\right)$ of Galois group
$\Gamma=G^{x}/\left(G^{x}\right)_{0}$.

(ii) The subalgebra $\mathcal{C}$ is invariant under the action of $\Gamma$ and $\Bbbk\left[G\left(\mathfrak{p}_{\mathrm{u}}\right)_{\mathrm{n}}\right]$
is the set of fixed points by $\Gamma$ in $\mathcal{C}$, where $G\left(\mathfrak{p}_{\mathrm{u}}\right)_{\mathrm{n}}$
is the normalization of $G\left(\mathfrak{p}_{\mathrm{u}}\right)$.\end{prop}
\begin{proof}
(i) Since $\Bbbk\left(G/\left(G^{x}\right)_{0}\right)^{\Gamma}=\Bbbk\left(G/G^{x}\right)$,
$\Bbbk\left(G/\left(G^{x}\right)_{0}\right)$ is a Galois extension
of Galois groupe $\Gamma$.

(ii) Since $\Bbbk\left[G\left(\mathfrak{p}_{\mathrm{u}}\right)\right]$
is invariant under the action of $\Gamma$, $\mathcal{C}$ is too.
Moreover, $\Bbbk\left[G\left(\mathfrak{p}_{\mathrm{u}}\right)_{\mathrm{n}}\right]$
is contained in $\Bbbk\left(G\left(\mathfrak{p}_{\mathrm{u}}\right)\right)=\Bbbk\left(G/G^{x}\right)$
the set of fixed points of $\Gamma$ in $\mathcal{C}$. Let $a$ be
in $\mathcal{C}$ such that $a$ is a fixed point by $\Gamma$. Then
$a$ is in $\Bbbk\left(G/G^{x}\right)$. Hence $a$ is in $\Bbbk\left[G\left(\mathfrak{p}_{\mathrm{u}}\right)_{\mathrm{n}}\right]$
since it is in $\mathcal{C}$, whence the assertion. 
\end{proof}

\begin{prop}
(i) The field $\Bbbk\left(G\times_{\mathbf{P}}\mathfrak{p}_{\mathrm{u}}^{k}\right)$
is the field of rational fractions\linebreak$\Bbbk\left(G\times_{\mathbf{P}}\mathfrak{p}_{\mathrm{u}}\right)\left(\tau_{1},\ldots,\tau_{m}\right)$
over $\Bbbk\left(G\times_{\mathbf{P}}\mathfrak{p}_{\mathrm{u}}\right)$,
where $m=\dim\mathfrak{p}_{\mathrm{u}}^{k-1}$. Moreover, it is a subfield of $\Bbbk\left(G/\left(G^{x}\right)_{0}\right)\left(\tau_{1},\ldots,\tau_{m}\right)$.

(ii) The field $\Bbbk\left(\mathcal{P}_{\mathrm{u}}^{\left(k\right)}\right)$
is the set of fixed points of $\Gamma$ in $\Bbbk\left(G/\left(G^{x}\right)_{0}\right)\left(\tau_{1},\ldots,\tau_{m}\right)$
under the trivial extension of its action on $\Bbbk\left(G/\left(G^{x}\right)_{0}\right)$.\end{prop}
\begin{proof}
(i) Since $\mathbf{P}_{-,\mathrm{u}}$ is isomorphic to an open subset
of $G/\mathbf{P}$ and since\linebreak $G\times_{\mathbf{P}}\mathfrak{p}_{\mathrm{u}}^{k}$
is a vector bundle over $G/\mathbf{P}$, $\mathbf{P}_{-,\mathrm{u}}\times\mathfrak{p}_{\mathrm{u}}^{k}$
is isomorphic to an open subset of $G\times_{\mathbf{P}}\mathfrak{p}_{\mathrm{u}}^{k}$.
Moreover, $G\times_{\mathbf{P}}\mathfrak{p}_{\mathrm{u}}^{k}$ is a vector bundle over $G\times_{\mathbf{P}}\mathfrak{p}_{\mathrm{u}}$ whose fibers are isomorphic to $\mathfrak{p}_{\mathrm{u}}^{k-1}$. Hence $\Bbbk\left(G\times_{\mathbf{P}}\mathfrak{p}_{\mathrm{u}}^{k}\right)$ equals $\Bbbk\left(G\times_{\mathbf{P}}\mathfrak{p}_{\mathrm{u}}\right)\left(\tau_{1},\ldots,\tau_{m}\right)$, where $m=\dim\mathfrak{p}_{\mathrm{u}}^{k-1}$. So it is a subfield of $\Bbbk\left(G/\left(G^{x}\right)_{0}\right)\left(\tau_{1},\ldots,\tau_{m}\right)$ since $\Bbbk\left(G/\left(G^{x}\right)_{0}\right)$ is an extension of $\Bbbk\left(G\times_{\mathbf{P}}\mathfrak{p}_{\mathrm{u}}\right)$.

(ii) Let $\pi$ be the first projection from $\mathcal{P}_{\mathrm{u}}^{\left(k\right)}\cap G\left(\mathfrak{p}_{\mathrm{u}}'\right)\times\mathfrak{g}^{k-1}$
to $G\left(\mathfrak{p}_{\mathrm{u}}'\right)$, let $x_{1}$ be in
$G\left(\mathfrak{p}_{\mathrm{u}}'\right)$ and let $z$ be in $\pi^{-1}\left(x_{1}\right)$
such that it is a smooth point of $\mathcal{P}_{\mathrm{u}}^{\left(k\right)}$
and $\pi$ is smooth at $z$. Then there exists a smooth affine open
subset $O$ containing $z$ in $\mathcal{P}_{\mathrm{u}}^{\left(k\right)}$ such that
$\pi_{O}$ the restriction of $\pi$ to $O$ is a submersion onto
an open subset of $G\left(\mathfrak{p}_{\mathrm{u}}'\right)$. Hence
the fiber of $\pi_{O}$ at $\pi\left(z\right)$ is an open subset
of an affine space. So the restrictions of the affine coordinates $\tau_{1},\ldots,\tau_{m}$ 
of this space to $\pi_{O}^{-1}\left(\pi\left(z\right)\right)$
are the restrictions of regular functions on $O$ which we denote
again $\tau_{1},\ldots,\tau_{m}$. Let $\psi$ be the morphism defined
by 
\[
\begin{array}{cccc}
\psi: & O & \longrightarrow & \pi\left(O\right)\times\Bbbk^{m}\\
 & y & \longmapsto & \left(\pi\left(y\right),\tau_{1}\left(y\right),\ldots,\tau_{m}\left(y\right)\right).
\end{array}
\]
 Then $\psi$ is a submersion at $z$. So, 
for all $y$ in an open subset $O'$ of $O$, containing $z$, the intersection of the kernels of the differential
at $y$ of the restrictions of $\tau_{1},\ldots,\tau_{m}$ to $F_{\pi\left(y\right)}$
the fiber of $\pi$ at $\pi\left(y\right)$ equals $\left\{ 0\right\}$
so that $\psi$ is locally injective on $F_{\pi\left(y\right)}$.
Let 
\[
\Delta:=\left\{ \left(x',x''\right)\in O'\times O'\mid\psi\left(x'\right)=\psi\left(x''\right)\right\} 
\]
 and let $\Delta_{O'}$ be the diagonal of $O'$. Hence the dimension
of $\Delta\cap F_{\pi\left(y\right)}\times F_{\pi\left(y'\right)}$
is zero, for $\left(y,y'\right)$ in $\Delta$. Then the dimensions
of $\Delta$ and of $O$ are equal. Hence $\Delta_{O'}$ is an irreducible
component of $\Delta$. Then there exists an open subset $O''$ of
$O$ such that the restriction of $\psi$ to $O''$ is injective. Hence
$\Bbbk\left(\mathcal{P}_{\mathrm{u}}^{\left(k\right)}\right)=
\Bbbk\left(G\left(\mathfrak{p}_{\mathrm{u}}\right)\right)\left(\tau_{1},\ldots,\tau_{m}\right)$, 
whence the assertion since $\Bbbk\left(G\left(\mathfrak{p}_{\mathrm{u}}\right)\right)=\Bbbk\left(G/\left(G^{x}\right)_{0}\right)^{\Gamma}$.
\end{proof}
Let $\mathcal{C}^{(k)}$ be the integral closure of $\Bbbk\left[\mathcal{P}_{\mathrm{u}}^{\left(k\right)}\right]$
in $\Bbbk\left(G/\left(G^{x}\right)_{0}\right)\left[\tau_{1},\ldots,\tau_{m}\right]$
and let $\left(\left(\mathcal{P}_{\mathrm{u}}^{\left(k\right)}\right)_{\mathrm{n}},\nu_{k}\right)$
be the normalization of $\mathcal{P}_{\mathrm{u}}^{\left(k\right)}$.
\begin{prop}
\label{prop:direct-fact}(i) The algebra $\Bbbk\left[\left(\mathcal{P}_{\mathrm{u}}^{\left(k\right)}\right)_{\mathrm{n}}\right]$
is the set of fixed points of $\Gamma$ in $\mathcal{C}^{(k)}$.

(ii) The algebra $\Bbbk\left[\left(\mathcal{P}_{\mathrm{u}}^{\left(k\right)}\right)_{\mathrm{n}}\right]$
is a direct factor of $\mathcal{C}^{(k)}$ as a $\Bbbk\left[\left(\mathcal{P}_{\mathrm{u}}^{\left(k\right)}\right)_{\mathrm{n}}\right]$-module.\end{prop}
\begin{proof}
(i) Since $\Bbbk\left[\mathcal{P}_{\mathrm{u}}^{\left(k\right)}\right]$
is invariant under the action of $\Gamma$, $\mathcal{C}^{(k)}$ is
too. Moreover, $\Bbbk\left[\left(\mathcal{P}_{\mathrm{u}}^{\left(k\right)}\right)_{\mathrm{n}}\right]$
is contained in $\Bbbk\left(\mathcal{P}_{\mathrm{u}}^{\left(k\right)}\right)$
the set of fixed points of $\Gamma$ in $\mathcal{C}^{\left(k\right)}$.
Let $a$ be in $\mathcal{C}^{(k)}$ such that $a$ is a fixed point
by $\Gamma$. Then $a$ is in $\Bbbk\left(\mathcal{P}_{\mathrm{u}}^{\left(k\right)}\right)$
and it verify an equation of integral dependence on $\Bbbk\left[\mathcal{P}_{\mathrm{u}}^{\left(k\right)}\right]$.
Hence $a$ is in $\Bbbk\left[\left(\mathcal{P}_{\mathrm{u}}^{\left(k\right)}\right)_{\mathrm{n}}\right]$,
whence the assertion. 

(ii) Let $\Phi$ be the map 
\[
\begin{array}{ccccc}
\Phi: & \mathcal{C}^{(k)} & \longrightarrow & \Bbbk\left[\left(\mathcal{P}_{\mathrm{u}}^{\left(k\right)}\right)_{\mathrm{n}}\right]\\
 & c & \longmapsto & c^{\#}=\frac{1}{|\Gamma|}\mbox{{\ensuremath{\sum\limits _{\gamma\in\Gamma}}}}\gamma.c & .
\end{array}
\]
 Then $\Phi$ is a projection from $\mathcal{C}^{(k)}$ onto $\Bbbk\left[\left(\mathcal{P}_{\mathrm{u}}^{\left(k\right)}\right)_{\mathrm{n}}\right]$
by (i). Since $\left(bc\right)^{\#}=bc^{\#}$, for $b\in\Bbbk\left[\left(\mathcal{P}_{\mathrm{u}}^{\left(k\right)}\right)_{\mathrm{n}}\right]$
and $c\in\mathcal{C}^{(k)}$, $\mathcal{C}^{(k)}$ is the direct sum
as $\Bbbk\left[\left(\mathcal{P}_{\mathrm{u}}^{\left(k\right)}\right)_{\mathrm{n}}\right]$-module
of $\Bbbk\left[\left(\mathcal{P}_{\mathrm{u}}^{\left(k\right)}\right)_{\mathrm{n}}\right]$
and $\mathrm{M}_{k}:=\mathrm{Ker}\Phi$.\end{proof}

\begin{prop} There exists a pure morphism $\psi$ in the sense of \cite{key-11} from $X^{\left(k\right)}$ to $\left(\mathcal{P}_{\mathrm{u}}^{\left(k\right)}\right)_{\mathrm{n}}$
such that $\varphi=\nu_{k}\circ\psi$.
\end{prop}
\begin{proof}
Since $X^{\left(k\right)}$ is normal and since $\varphi$ is surjective,
there exists a morphism $\psi$ from $X^{\left(k\right)}$ to $\left(\mathcal{P}_{\mathrm{u}}^{\left(k\right)}\right)_{\mathrm{n}}$
such that $\varphi=\nu_{k}\circ\psi$ (\cite{key-10}, ch. II, Ex.
3.8). By Proposition \ref{prop:direct-fact}, 
\[
\Bbbk\left[X^{\left(k\right)}\right]=\Bbbk\left[\left(\mathcal{P}_{\mathrm{u}}^{\left(k\right)}\right)_{\mathrm{n}}\right]\oplus\left(\mathrm{M}_{k}\cap\Bbbk\left[X^{\left(k\right)}\right]\right).
\]
 Hence $\psi$ is a pure morphism in the sense of \cite{key-11}.\end{proof}

\end{document}